\documentclass{azjmPRE_ArXiv}

\usepackage[
unicode,colorlinks,linkcolor=blue,
citecolor=red,bookmarksopen,pdfhighlight=/N]{hyperref}

\begin{document}

\title{The Logarithm of the Modulus of an Entire Function as a Minorant for a Subharmonic Function outside a Small Exceptional Set}

\author[B. Khabibullin]{B.\,N. Khabibullin\footnote{This  work was supported by a Grant of the Russian Science Foundation (Project No. 18-11-00002).}}

\begin{abstract}
Let $u\not\equiv -\infty$ be a subharmonic function on the complex plane $\mathbb C$. In 2016, we obtained a result on the existence of an entire function $f\neq 0$ satisfying the estimate $\log|f|\leq {\sf B}_u$ on $\mathbb C$, where functions ${\sf B}_u$ are integral averages of $u$ for rapidly shrinking disks as it approaches infinity. We give another equivalent version of this result with $\log |f|\leq u$ outside a very  small exceptional set  if $u$ is of finite order.    
\end{abstract}

\keywords{subharmonic function, entire function, exceptional set, Riesz measure, integral average, covering of sets, type and order of function.}
\ams{30D20, 31A05, 26A12}

\maketitle

\section{Introduction}
\subsection{Definitions and notations. Preliminary result}

We consider the  set\/ $\mathbb R$ of {\it real numbers\/} mainly as the \textit{real axis\/} in the {\it complex plane\/} $\mathbb C$, and  $\mathbb R^+:=\{x\in \mathbb R\colon x\geq 0\}$  is the  {\it positive semiaxis\/} in $\mathbb C$. Besides, $\overline{\mathbb R}:=\mathbb R\cup \{\pm\infty\}$ is {\it the extended real line\/}  with the natural order
$-\infty \leq x \leq +\infty$ for every $x\in \overline{\mathbb R}$, $\overline{\mathbb R}^+:={\mathbb R}^+\cup\{+\infty\}$, $x^+:=\sup\{x,0\}$ for each $x\in \overline{\mathbb R}$.  For an {\it extended real function\/} $f\colon S\to \overline{\mathbb R}$, its {\it positive part\/} is the function 
$
f^+\colon s\underset{{\text{\tiny $s\in S$}}}{\longmapsto} \bigl(f(s)\bigr)^+.
$
 
For $z\in \mathbb C$ and $r\in \mathbb R^+$, we denote by 
$D(z,r):=\{z' \in \mathbb C  \colon |z'-z|<r\}$ {\it  the open disk 
 centered at $z$ and of radius $r$,} where $D(z,0)$ is the empty set $\varnothing$,   
 $D(r):=D(0,r)$, $\overline{D}(z,r):=\{z' \in \mathbb C  \colon |z'-z|\leq r\}$ {\it the closed  disk centered at $z$ and of radius $r$,} $\overline D(r):=\overline D(0,r)$, and 
$\partial \overline D(z,r):=\overline{D}(z,r)\!\setminus\!D(z,r)$
 {\it the circle centered at $z$ and of radius $r$,} 
$\partial \overline D(r):=\partial \overline D(0,r)$.
For a function $v\colon \overline   D(z,r)\to \overline{\mathbb R}$, 
we define the {\it integral averages\/} on circles and  disks as 
\begin{subequations}\label{df:MCB}
\begin{align}
\mathsf{C}_v(z, r)&
:=\frac{1}{2\pi} \int_{0}^{2\pi}  v(z+re^{i\theta}) {\rm \, d} \theta, \quad &\mathsf{C}^{\text{\rm \tiny rad}}_v(r):=\mathsf{C}_v(0, r),\tag{\ref{df:MCB}C}\label{df:MCBc}\\
\mathsf{B}_v(z,r)&
:=\frac{2}{r^2}\int_{0}^{r}\mathsf{C}_v(z, t)t{\rm \, d} t ,
\quad &\mathsf{B}^{\text{\rm \tiny rad}}_v(r):=\mathsf{B}_v(0, r); \tag{\ref{df:MCB}B}\label{df:MCBb}\\
 \mathsf{M}_v(z,r)&
:=\sup_{z'\in \partial\overline D(z,r)}v(z'), 
\quad &\mathsf{M}^{\text{\rm \tiny rad}}_v(r):=\mathsf{M}_v(0, r),  
\tag{\ref{df:MCB}M}\label{df:MCBm}
\end{align}
\end{subequations}
where $\mathsf{M}_v(z,r):=\sup\limits_{z'\in \overline D(z,r)}v(z')$
if $v$ is subharmonic  on $\mathbb C$ \cite[Definition 2.6.7]{Rans}, \cite{HK}.  

The following result \cite[Corollary 2]{KhaBai16} of 2016 found several useful applications
\cite[Lemma 5.1]{BaiKhaKha17}, \cite{BaiKha16}, \cite[Proposition 2]{KhaKha16}, \cite{BalKha17}, \cite[Lemma 6.3]{KhaShm19}, \cite{KhaKha19}, \cite[7.1]{KhaRozKha19} for entire functions on the complex plane:

\begin{theorem}[{\rm \cite[Corollary 2]{KhaBai16}, see also \cite[Lemma 5.1]{BaiKhaKha17}}]\label{pr:f}
Let $u\not\equiv -\infty$ be a subharmonic function on\/ $\mathbb C$, and    $q\in \mathbb R^+$  be a number with the corresponding function
\begin{equation}\label{p}
Q\colon z\underset{\text{\tiny $z\in \mathbb C$}}{\longmapsto} \frac{1}{(1+|z|)^q}\leq 1.   
\end{equation}
Then there is an entire function $f_q\not\equiv 0$ on $\mathbb C$ such that 
\begin{equation}\label{fBC}
 \log \bigl|f_q(z)\bigr|\leq {\sf B}_{u}\bigl(z,Q(z)\bigr)
\leq {\sf C}_{u}\bigl(z,Q(z)\bigr)\leq 
{\sf M}_{u}\bigl(z,Q(z)\bigr) \quad\text{for each $z\in \mathbb C$}.
\end{equation}
\end{theorem}
In this article, we obtain another equivalent version of Theorem \ref{pr:f} for subharmonic functions of finite order. This version may be useful in another situations that we are not discussing here.

\subsection{Main result for minorants outside an  exceptional set} 
For an extended real function  $m\colon \mathbb R^+\to \overline{\mathbb R}$, we define \cite{Kiselman}, \cite{Azarin}, \cite[2.1, (2.1t)]{KhaShm19}
\begin{equation}\label{order}
{\sf ord}[m]:=\limsup_{r\to +\infty} \frac{\log\bigl(1+m^+(r)\bigr)}{\log r}\in 
\overline{\mathbb R}^+,
\end{equation}
the {\it order of growth of\/} $m$; for $p\in \mathbb R^+$, 
\begin{equation}\label{typevf}
{\sf type}_p[m]:=\limsup_{r\to +\infty} \frac{m^+(r)}{r^p}\in 
\overline{\mathbb R}^+, 
\end{equation}
the {\it type of growth of $m$ at the order $p$.\/} Thus, it is easy to see that
\begin{equation}\label{mot}
{\sf order} [m]=\inf\bigl\{p\in \mathbb R^+\colon 
{\sf type}_p [m]<+\infty\bigr\}, \quad \inf\varnothing :=+\infty.
\end{equation}

If $u$ is a subharmonic function on $\mathbb C$,  then 
\begin{equation}\label{otu}
{\sf order}[u]\overset{\eqref{df:MCBm}}{:=}{\sf order}[{\sf M}^{\text{\rm \tiny rad}}_u], \quad 
{\sf type}_p[u]\overset{\eqref{df:MCBm}}{:=}{\sf type}_p[{\sf M}^{\text{\rm \tiny rad}}_u],
\end{equation}
and, under the condition  ${\sf type}_p[u]<+\infty$,    
the following $2\pi$-periodic function 
\begin{equation}\label{ind}
{\sf ind}_p[u](s):=\limsup_{r\to +\infty}\frac{u(re^{is})}{r^p}\in \mathbb R, \quad s\in \mathbb R, 
\end{equation}
is called the {\it indicator of the growth of $u$ at the order} $p$  \cite[3.2]{Azarin}.

For a ray or a circle  on $\mathbb C$, 
we denote by ${\sf mes}$  the {\it linear Lebesgue measure\/} on this ray or the  {\it measure of length\/} on this circle.

\begin{theorem}\label{thm} Let $u\not\equiv -\infty$ be a subharmonic function on the complex plane, and     ${\sf order}[{\sf B}^{\text{\rm \tiny rad}}_u]\overset{\eqref{df:MCBb}}{<}+\infty$. Then the conclusion \eqref{p}--\eqref{fBC} of Theorem\/ {\rm \ref{pr:f}} with arbitrary positive numbers $q\in \mathbb R^+$ is equivalent to the following statement:

For any positive  $q\in \mathbb R^+$, there are  
an entire function $f_q\not\equiv 0$ and a no-more-than countable  set of disks $D(z_k,t_k)$, $k=1,2,\dots$,  
such that
\begin{subequations}\label{E}
\begin{align}
\log \bigl|f_q(z)\bigr|&\leq u(z)
\quad\text{for each $z\in  \mathbb C\setminus E_q$, where}
\tag{\ref{E}I}\label{EI}\\
E_q:=\bigcup_k D(z_k,r_k),& \quad
\sup_k t_k\leq 1, \quad \sum_{|z_k|\geq R} t_k=O\Bigl(\frac{1}{R^q}\Bigr) \text{ as  $R\to +\infty$}.
\tag{\ref{E}E}\label{EE}
\end{align}
\end{subequations}
If ${\sf ord}[u]\overset{\eqref{otu}}{<}+\infty$, then statements  \eqref{p}--\eqref{fBC} of Theorem\/ {\rm \ref{pr:f}} 
 or statements \eqref{E} of this Theorem\/ {\rm \ref{thm}} can be supplemented by the following restrictions:
\begin{subequations}\label{ot}
\begin{align}
{\sf ord}\bigl[\log |f_q|\bigr]&\overset{\eqref{order},\eqref{mot},\eqref{otu}}{\leq} {\sf ord}[u],
\tag{\ref{ot}o}\label{oto}
\\
{\sf type}_p\bigl[\log |f_q|\bigr]&\overset{\eqref{typevf},\eqref{otu}}{\leq} {\sf type}_p[u]\quad \text{for each $p\in \mathbb R^+$},\tag{\ref{ot}t}\label{ott}\\
{\sf ind}_p \bigl[\log |f_q|\bigr]&\overset{\eqref{ind}}{\leq} {\sf ind}_p[u]\quad \text{for each $q\in \mathbb R^+$}.\tag{\ref{ot}i}\label{oti} 
\end{align}
\end{subequations}
Besides, for  any  ray $L\subset \mathbb C$, we have
\begin{equation}\label{R}
{\sf mes}\Bigl(L\setminus \bigl( E_q\cup D(R)\bigr)\Bigr)=O\Bigl(\frac{1}{R^q}\Bigr)
\quad \text{as  $R\to +\infty$},
\end{equation}
 and also 
\begin{equation}\label{C}
{\sf mes}\Bigl( E_q\bigcap \partial \overline D(R)\Bigr)=O\Bigl(\frac{1}{R^q}\Bigr)
\quad \text{as  $R\to +\infty$}.
\end{equation}
\end{theorem}

 Theorem \ref{thm} is proved in Sec.~\ref{pr:Th} after some preparation.

\section{Preparatory results}

\subsection{On exceptional sets}\label{S:ex} 

For a Borel measure $\mu$ on $\mathbb C$, we set
\begin{equation}\label{mzt}
\mu(z,t):=\mu\bigl(\overline D(z,t)\bigr),\quad z\in \mathbb C, \; t\in \mathbb R^+. 
\end{equation}

For  a function $d\colon \mathbb C \to \mathbb R^+$, 
$S\subset \mathbb C $ and $r\colon \mathbb C \to \mathbb R$,  we define 
\begin{equation*}\label{Sdr}
\begin{split}
S^{\cup d}&:=\bigcup_{z\in S}D\bigl(z,d(z)\bigr)\subset \mathbb C , \\
r^{\vee d}&\colon z\underset{\text{\tiny $z\in \mathbb C$}}{\longmapsto} \sup 
\Bigl\{ r(z')\colon z'\in D\bigl(z,d(z)\bigr)\Bigr\}\in \overline{\mathbb R},
\end{split}
\end{equation*}
and denote  \textit{the indicator function\/} of set $S$ 
by 
$$
{\bf 1}_S\colon z\underset{\text{\tiny $z\in \mathbb C$}}{\longmapsto} \begin{cases}
1&\text{ if $z\in S$},\\
0&\text{ if $z\notin S$}.
\end{cases}
$$

\begin{lemma}[{\rm cf. \cite[Normal Points Lemma]{Kha84}, 
\cite[\S~4. Normal points, Lemma]{KudKha07}}]\label{pr:r} Let  
$r\colon \mathbb C \to \mathbb R^+$ be a Borel function such that 
\begin{equation}\label{r0}
d:=2\sup\{ r(z)\colon z\in \mathbb C \}
 <+\infty,
\end{equation} 
and $\mu$ be a Borel positive measure on $\mathbb C$ with  
\begin{equation}\label{Er}
 E_{\mu,r}:= \left\{z\in \mathbb C  \colon \int_0^{r(z)}\frac{\mu(z,t)}{t}\,{\rm d} t>1\right\}\subset \mathbb C .
\end{equation}
Then there exists  a no-more-than countable set of disks $D(z_k,t_k)$, $k=1,2,\dots$,  such that 
\begin{equation}\label{DE}
\begin{split}
z_k\in E_{\mu,r}, \quad &t_k\leq r(z_k), \quad
E_{\mu,r}\subset \bigcup_k D(z_k,t_k), \\
\sup_{z\in \mathbb C }\#\bigl\{k&\colon z\in D(z_k,t_k)\bigr\}\leq 2020,
\end{split}
\end{equation}
i.\,e., the multiplicity of this covering $\{D(z_k,t_k)\}_{k=1,2,\dots}$ of set $E_{\mu,r}$ is not more than\/ $2020$, and, for every $\mu$-measurable subset $S\subset \bigcup_k D(z_k,t_k)$,
\begin{equation}\label{iSD}
\frac{1}{2020}\sum_{S\cap D(z_k,t_k)\neq \varnothing} t_k\leq \int_{S^{\cup d}} r^{\vee r} \,{\rm d} \mu \leq \int_{S^{\cup d}} r^{\vee d} \,{\rm d} \mu.
\end{equation}
\end{lemma}

\begin{proof}  By definition \eqref{Er},  there is a number 
\begin{equation}\label{tz}
 t_z\in \bigl(0,r(z)\bigr)\quad  \text{\it such that\/}\quad   0<t_z<r(z)\mu(z,t_z)
\text{ \it for each\/ } z\in E_{\mu,r}.
\end{equation}
Thus, the system $\mathcal D=\bigl\{D(z,t_z)\bigr\}_{z\in E}$ of these disks  has properties
\begin{equation*}
E_{\mu,r}\subset \bigcup_{z\in E} D(z,t_z), \quad 0<t_z\leq r(z)\overset{\eqref{r0}}{\leq} R. 
\end{equation*}
By the Besicovitch Covering Theorem 
\cite[2.8.14]{Federer}, \cite{FL}, \cite{GriKri10}, \cite[I.1, Remarks]{Gusman}, \cite{Sullivan}, \cite{Krantz} in the Landkof version \cite[Lemma 3.2]{L},  we can select some no-more-than countable subsystem in $\mathcal D$ of disks $D(z_k,t_k)\in \mathcal D$, $k=1,2,\dots$, $t_k:=t_{z_k}$, such that properties \eqref{DE} are fulfilled.
Consider a $\mu$-measurable subset $S\subset \bigcup_k D(z_k,t_k)$. 
In view of \eqref{tz} it is easy to see that
\begin{equation}\label{S}
\bigcup\Bigl\{D(z_k, t_k)\colon S\cap D(z_k,t_k)\neq \varnothing\Bigr\} 
\overset{\eqref{tz},\eqref{r0}}{\subset}\bigcup_{z\in S}D(z, d)
=S^{\cup d}.
\end{equation}
Hence, in view of \eqref{tz} and \eqref{DE}, 
we obtain 
\begin{multline*}
\sum_{S\cap D(z_k,t_k)\neq \varnothing} t_k:=\sum_{S\cap D(z_k,t_k)\neq \varnothing} t_{z_k} \overset{\eqref{tz}}{\leq}
\sum_{S\cap D(z_k,t_k)\neq \varnothing} r(z_k)\mu(z,t_k) \\
=\sum_{S\cap D(z_k,t_k)\neq \varnothing} \int_{D(z_k,t_k)}r(z_k)
{\rm \,d} \mu(z)
\overset{\eqref{tz}}{\leq}
\sum_{S\cap D(z_k,t_k)\neq \varnothing} \int_{D(z_k,t_k)}r^{\vee r}
{\rm \,d} \mu\\
\overset{\eqref{S}}{=}
\sum_{S\cap D(z_k,t_k)\neq \varnothing} \int_{S^{\cup d}}
{\bf 1}_{D(z_k,t_k)} r^{\vee r} {\rm \,d} \mu\\=
 \int_{S^{\cup d}}
\left(\sum_{S\cap D(z_k,t_k)\neq \varnothing}{\bf 1}_{D(z_k,t_k)}\right) r^{\vee r} {\rm \,d} \mu
\\
\overset{\eqref{DE}}{\leq} 
2020 \int_{S^{\cup d}}
 r^{\vee r} {\rm \,d} \mu\overset{\eqref{r0}}{\leq}
2020 \int_{S^{\cup d}}
 r^{\vee d} {\rm \,d} \mu.
\end{multline*} 
Thus, we obtain  \eqref{iSD}. This completes the proof of Lemma \ref{pr:r}.
\end{proof}

\begin{lemma}\label{lem1} Let $\bigl\{D(z_j, t_j)\bigr\}_{j\in J}$ 
be a system of disks in $\mathbb C$, $d:=2\sup_{j\in J} t_j<+\infty$.
Then, for each $z\in \mathbb C$, there is  a positive number $r\leq d$ such that 
\begin{equation}\label{Demp}
 \bigcup_{j\in J} D(z_j,t_j) \bigcap\partial \overline D(z,r)=\varnothing.
\end{equation} 
\end{lemma}
\begin{proof}  Consider a disk $\overline D(z,d)$, where, without loss of generality, we can assume that $z=0$.  Then, by condition $d:=2\sup_{j\in J} t_j<+\infty$, the union 
\begin{equation}\label{prdb}
 \bigcup_{j\in J}  \bigl(D(z_j,t_j)e^{-i\arg z_j}\Bigr)\bigcap \bigl[0,d] 
\end{equation}
of radial projections  $\bigl(D(z_j,t_j)e^{-i\arg z_j}\Bigr)\bigcap \bigl[0,d]$
of $D(z_j,t_j)$ onto radius $[0,d]$  is not empty, i.\,e. there is a point $r\in [0,d]$ outside \eqref{prdb}, which gives  \eqref{Demp} for $z=0$.   

Lemma \ref{lem1} is proved.
\end{proof} 

A consequence of Lemma \ref{lem1} is the following lemma:
\begin{lemma}\label{lemR} 
Let $\bigl\{D(z_k, t_k)\}_{k=1,2,\dots}$ be a system  of disks satisfying  \eqref{EE} with  a strictly positive  number $q\in \mathbb R^+\!\setminus\!\{0\}$,
and  $q'<q$ be a positive number. Then there exists a number $R_q\in \mathbb R^+$ such that for any $z\in \mathbb C$ with $|z|>R_q$ there is a positive number $r\leq \bigl(1+|z|\bigr)^{-q'}$ such that \eqref{Demp} holds for $J=\{1,2,\dots\}$.
\end{lemma}
\begin{proof}  By condition \eqref{EE}, there is a constant $C\in \mathbb R^+$ such that 
\begin{equation}\label{dDz}
\sum_{D(z_k,t_k)\!\setminus\!D(|z|-2)\neq \varnothing} t_k\leq \frac{C}{(1+|z|)^q}
\quad \text{for each $z\in \mathbb C$ with $|z|\geq 3$,}
\end{equation}
and, for $|z|\geq 3$, 
\begin{equation}\label{2}
\text{\it if $D(z_k,t_k)\setminus D(z,2)\neq \varnothing$, then $D(z_k,t_k)\!\setminus\!D(|z|-2)\neq \varnothing$.}
\end{equation}  
For $0\leq q'<q$, we choose $R_q\geq 3$ so that
\begin{equation}\label{R0}
C(1+|z|)^{q'-q}\leq \frac{1}{2}\quad\text{for all $|z|\geq R_q\geq 3$}.
\end{equation}
 It is follows from  \eqref{dDz}--\eqref{R0} that 
\begin{equation*}
\sum_{D(z_k,t_k)\!\setminus\!D(z, 2)\neq \varnothing} t_k\leq \frac{C}{(1+|z|)^q}\\
\overset{\eqref{R0}}{\leq} \frac{1}{2}
\frac{1}{(1+|z|)^{q'}}
\quad \text{for each $z\in \mathbb C$ with $|z|\geq R_q$,}
\end{equation*} 
and 
\begin{equation}\label{dDz+}
\sup_{D(z_k,t_k)\!\setminus\!D(z, 2)\neq \varnothing} t_k\leq  \frac{1}{2}
\frac{1}{(1+|z|)^{q'}}
\quad \text{for each $z\in \mathbb C$ with $|z|\geq R_q$.}
\end{equation} 

For an {\it arbitrary fixed point\/} $z\in \mathbb C$ with $|z|\geq R_q$, we consider   \begin{equation*}
J:=\bigl\{k\colon D(z_k,t_k)\!\setminus\!D(z,2)\neq \varnothing\bigr\}, \quad 
{\mathcal D}:=\bigl\{D(z_k, t_k)\bigr\}_{k\in J}.
\end{equation*} 
By Lemma \ref{lem1}, with these $J$ and $\mathcal D$ there is a circle
$\partial \overline D(z,r)$   such that 
\begin{equation*}
0\leq r\overset{\eqref{dDz+}}{\leq} (1+|z|)^{-q'}\leq 1, \quad  
\bigcup_{k\in J} D(z_k, t_k)\bigcap \partial \overline D(z,r)=\varnothing.
\end{equation*} 
But, in view of \eqref{2},   if $k\notin J$, then, as before, 
$D(z_k, t_k)\bigcap \partial \overline D(z,r)=\varnothing$. 

Lemma \ref{lemR} is proved.
\end{proof}

\subsection{The order and the upper density for measures on $\mathbb C$}

For  a Borel positive measure  $\mu$ on $\mathbb C$,  functions 
\begin{equation}\label{mu}
\mu^{\text{\rm \tiny rad}}\colon r\underset{\text{\tiny $r\in \mathbb R^+$}}{\overset{\eqref{mzt}}{\longmapsto}} \mu(0,r),  
\end{equation} 
is called the {\it radial counting function of $\mu$}, the quality 
$$
{\sf ord}[\mu]\overset{\eqref{order},\eqref{mot}}{:=}{\sf ord}\bigl[\mu^{\text{\rm \tiny rad}}\bigr]
$$ 
is called the \textit{order\/}  of  measure $\mu$, and, for $p\in \mathbb R^+$, the quantity 
\begin{equation}\label{tmu}
{\sf type}_p[\mu]\overset{\eqref{typevf}}{:=}{\sf type}_p\bigl [\mu^{\text{\rm \tiny rad}}\bigr]
\end{equation}
is called the \textit{upper density\/}  of  measure $\mu$ {\it at the order\/} $p$. 
 
If  $u\not\equiv -\infty$ is a subharmonic function  on $\mathbb C$ with the {\it Riesz measure} 
\begin{equation}\label{Rm}
\varDelta_u =\frac{1}{2\pi}\!\bigtriangleup\!u,
\end{equation}
where  the {\it Laplace operator\/} $\bigtriangleup$ acts in the sense of the theory of distributions or generalized functions \cite{Rans}, \cite{HK}, then,
by the Poisson\,--\,Jensen formula \cite[4.5]{Rans}, \cite{HK}
\begin{equation}\label{PJ}
u(z)={\sf C}_u(z,r)-\int_0^r\frac{\varDelta_u(z,t)}{t}\,{\rm d}t, \quad z\in \mathbb C, 
\end{equation}
 in a disk $D(z,r)$  in the form \cite[3, (3.3)]{BaiKhaKha17}
\begin{equation*}
{\sf C}_u(r)-{\sf C}_u(1)=\int_1^r\frac{\varDelta_u^{\text{\rm \tiny rad}}(t)}{t}\,{\rm d}t, 
\end{equation*}
and by \eqref{df:MCBb} together with 
\begin{lemma}[{\cite{Beardon}, \cite[Theorem 3]{FM}}]\label{CB}
If $u$ be a subharmonic function on $\mathbb C$, then
${\sf B}(z,t)\leq {\sf C}(z,t)\leq {\sf B}(z,\sqrt{e}t)$ for each $z\in \mathbb C$ and for each  $t\in \mathbb R^+$.
\end{lemma}
we can easily obtain 
\begin{lemma}\label{w}
Let $u\not\equiv -\infty$ be a subharmonic function  on $\mathbb C$ with  Riesz measure \eqref{Rm}. Then, for each $r\geq 1$,  
\begin{equation}\label{dC}
{\sf B}_u(r)-{\sf C}_u(1)\leq {\sf C}_u(r)-{\sf C}_u(1)\leq 
\int_1^r\frac{\varDelta_u^{\text{\rm \tiny rad}}(t)}{t}\,{\rm d}t\leq 
{\sf C}_u(r)\leq {\sf B}_u(\sqrt{e}r). 
\end{equation}
In particular,  we have the  following equalities\/ 
$${\sf ord}[\varDelta_u]={\sf ord}[{\sf C}_u]={\sf ord}[{\sf B}_u],
$$ and
the following equivalences 
\begin{equation*}
\bigl[{\sf type}_p[\varDelta_u]<+\infty\bigr]
\Longleftrightarrow\bigl[{\sf type}_p[{\sf C}_u]<+\infty\bigr]\Longleftrightarrow\bigl[{\sf type}_p[{\sf B}_u]<+\infty\bigr]
\end{equation*}
 for each strictly positive $p\in \mathbb R^+\!\setminus\!\{0\}$.
\end{lemma}

\section{The proof of Theorem \ref{thm}}\label{pr:Th}

\subsection{From Theorem \ref{pr:f} to \eqref{E}}
Let $q'\in \mathbb R^+$.  
By Lemma  \ref{w}, we have 
\begin{equation}\label{au}
a_{u}:={\sf ord}[\varDelta_u]\overset{\eqref{dC}}{=}{\sf ord}[{\sf C}_u]<+\infty. 
\end{equation}
 We choose  
\begin{equation}\label{pn}
q:=a_u+q'+3\geq 3.
\end{equation} 
and an entire function $f_q$ from Theorem \ref{pr:f} with properties \eqref{p}--\eqref{fBC}.  Then, for entire function $e^{-1}f_q\not\equiv 0$, we obtain
\begin{multline}\label{efp}
\log \bigl|e^{-1}f_q(z)\bigr|\leq {\sf C}_{u}\bigl(z,Q(z)\bigr)-1\\
\overset{\eqref{PJ}}{=}u(z)+\int_0^{Q(z)}\frac{\varDelta_u(z,t)}{t}\,{\rm d} t-1\quad\text{for each $z\in \mathbb C\!\setminus\! (-\infty)_u$}, 
\end{multline}
where $(-\infty)_u:=\bigl\{z\in \mathbb C\colon u(z)=-\infty\bigr\}$ is a {\it minus-infinity\/} $G_{\delta}$ {\it polar set\/} \cite[3.5]{Rans}, and
$1$-dimensional Huasdorff measure   of $(-\infty)_u$ is zero \cite[5.4]{HK}. Therefore, 
this set $(-\infty)_u$ can be covered by a system of disks  as in \eqref{EE}
with $q'$ instead of $q$. By Lemma \ref{pr:r} with  
\begin{equation}\label{Ep}
r\overset{\eqref{p}}{:=}Q, \quad d\overset{\eqref{r0}}{\leq}2,\quad  
 \mu\overset{\eqref{Rm}}{:=}\varDelta_u,  \quad E_{q}
\overset{\eqref{DE}}{:=}\bigcup_k D(z_k,t_k)
 \overset{\eqref{Er},\eqref{EE}}{\supset}E_{\mu,r}, 
\end{equation}
 we have, in view of  \eqref{efp}, 
\begin{equation}\label{efpE}
\log \bigl|e^{-1}f_q(z)\bigr|\overset{\eqref{efp},\eqref{Er}}{\leq}u(z), \quad\text{for each $z\in \mathbb C\!\setminus\! \bigl(E_q\cup (-\infty)_u\bigr)$}.
\end{equation}
If $S:=E_q\!\setminus\! D(R)$  and $R\geq 4$, then, by \eqref{iSD},   
\begin{multline*}\label{SD}
\frac1{2020}\sum_{|z_k|\geq R} t_k\overset{\eqref{iSD}}{\leq} 
\int_{S^{\cup d}} r^{\vee d} \,{\rm d} \varDelta_u
\overset{\eqref{Ep}}{\leq} \int_{|z|\geq R-2}
\frac{1}{\bigl(1+(|z|-2)\bigr)^q}\,{\rm d}\varDelta_u(z)
\\
=\int_{R-2}^{+\infty} \frac{1}{(t-1)^q}\,{\rm d}\varDelta^{\text{\rm \tiny rad}}_u(t)
\overset{\eqref{pn}}{\leq}
\int_{R-2}^{+\infty} \frac{\varDelta^{\text{\rm \tiny rad}}_u(t)}{(t-1)^{q-1}}\,{\rm d}t\\
\overset{\eqref{mot},\eqref{au}}{\leq} {\rm const}\int_{R-2}^{+\infty} \frac{t^{a_u+1}}{(t-1)^{q-1}}\,{\rm d}t
\overset{\eqref{pn}}{=}O(R^{a_u+3-q})\quad\text{as $R\to +\infty$},  
\end{multline*}
where ${\rm const}\in \mathbb R^+$ is independent of $R$, and
 $R^{a_u+3-q}\overset{\eqref{pn}}{=}R^{-q'}$. 
The latter together with \eqref{efpE} gives the statements \eqref{E} of Theorem \ref{thm}. 

\subsection{From \eqref{E}  to Theorem \ref{pr:f}}
Let $q^*\in \mathbb R^+$. Suppose that statements \eqref{E} of Theorem \ref{thm} are fulfilled with $q>q'>q^*\geq 0$. By Lemma \ref{lemR} there exists a number $R_q\in \mathbb R^+$ such that for any $z\in \mathbb C$ with $|z|>R_q$ there is a positive number $r_z\leq \bigl(1+|z|\bigr)^{-q'}$ such that 
$E_q\cap \partial \overline D(z,r_z)=\varnothing$.  
Hence, by \eqref{EI}, we obtain
\begin{equation}\label{q}
\log \bigl|f_q(z+r_ze^{is})\bigr|\leq u(z+r_ze^{is})\quad\text{for each $s\in  \mathbb R$}
\end{equation} 
and for any $z\in \mathbb C$ with $|z|\geq R_q$. Therefore,
\begin{equation*}
\log \bigl|f_q(z) \bigr|\leq {\sf C}_{\log|f_q|}(r_z)\overset{\eqref{q}}{\leq} {\sf C}_{u}(r_z)
\leq {\sf C}\Bigl(z,\frac{1}{(1+|z|)^{q'}}\Bigr)
\quad\text{if $|z|\geq R_q$}. 
\end{equation*}
Hence there exist a sufficiently small number $a>0$ and a sufficiently large number 
$R_{q^*}\geq R_q $ such that 
\begin{equation*}
\log \bigl|af_q(z) \bigr|
\leq 
{\sf C}\Bigl(z,\frac{1}{\sqrt{e}(1+|z|)^{q^*}}\Bigr)
\quad\text{if   $|z|\geq R_{q^*}$}.
\end{equation*}
The function $\log|af_q|$ is bounded from above on $D(R_{q^*})$, and the function
\begin{equation*}
{\sf C}\Bigl(\cdot , \frac{1}{\sqrt{e}(1+R_{q^*})^{q^*}}\Bigr) \colon z
\underset{\text{\tiny $z\in \mathbb C$}}{\longmapsto} 
{\sf C}\Bigl(z, \frac{1}{\sqrt{e}(1+R_{q^*})^{q^*}}\Bigr) 
\end{equation*}
is continuous \cite[Theorem 1.14]{Helms}. Therefore, there exists a sufficiently small number $b>0$ such that  
\begin{equation*}
\log \bigl|abf_q(z) \bigr|
\leq 
{\sf C}\Bigl(z,\frac{1}{\sqrt{e}(1+|z|)^{q^*}}\Bigr)
\quad\text{for all   $z\in \mathbb C$}.
\end{equation*}
Hence, for $f_{q^*}:=abf_q\neq 0$, by Lemma \ref{CB}, we obtain  \eqref{fBC} with $q^*\in \mathbb R^+$ instead of $q$ in \eqref{p}. 
Further, equalities \eqref{oto} and \eqref{ott} for orders and types are obvious consequences of  \eqref{fBC} even for $q = 0$. Similarly, we obtain equality \eqref{oti}, since indicators \eqref{ind} of the growth of $\log|f_q|$ and $u$ are continuous.
Relations \eqref{R}--\eqref{C} are obvious particular cases of \eqref{EE}.

\begin{footnotesize}
\begin{flushleft}
Bulat N. Khabibullin\\
\textit{Bashkir State University, Ufa, Russian Federation}\\
\textit{E-mail:} khabib-bulat@mail.ru\\
\end{flushleft}
\end{footnotesize}

\end{document}